\documentclass[a4paper]{amsart}
\pdfoutput=1 
\usepackage[utf8]{inputenc}
\usepackage[T1]{fontenc}
\usepackage{lmodern}
\usepackage{amssymb}
\usepackage{amsmath}
\usepackage{mathtools}
\usepackage[lite]{amsrefs}

\usepackage{enumitem}
\setlist[enumerate]{font=\textup}

\usepackage{microtype}
\usepackage[pdftitle={A characterization of simplicity of reduced groupoid C*-algebras},
pdfauthor={Paolo Boldrini},
pdfsubject={Mathematics},hidelinks]{hyperref}

\newtheorem{introthm}{Theorem}

\newtheorem{introcor}[introthm]{Corollary}

\newtheorem{thm}{Theorem}[section]
\newtheorem{lem}[thm]{Lemma}
\newtheorem{prop}[thm]{Proposition}
\newtheorem{cor}[thm]{Corollary}

\theoremstyle{definition}
\newtheorem{defn}[thm]{Definition}
\newtheorem{qst}[thm]{Question}
\newtheorem*{nota}{Notation}
\theoremstyle{remark}

\newcommand{\defeq}{\vcentcolon=}
\def\acts{\curvearrowright}
\def\inv{^{-1}}
\newcommand{\N}{\mathbb N}
\newcommand{\Z}{\mathbb{Z}}

\newcommand{\C}{\mathbb{C}}
\newcommand{\cs}{\mathrm{C}^*}
\newcommand{\cA}{\mathcal{A}}
\newcommand{\cG}{\mathcal{G}}
\newcommand{\cH}{\mathcal{H}}
\DeclareMathOperator{\Ad}{Ad}
\DeclareMathOperator{\Sub}{Sub}
\DeclareMathOperator{\Iso}{Iso}
\DeclareMathOperator{\Ind}{Ind}
\DeclareMathOperator{\Am}{Am}
\newcommand*{\id}{\mathrm {id}}

\title[A characterization of simplicity of reduced groupoid C*-algebras]{A characterization of simplicity of reduced groupoid C*-algebras}
\author{Paolo Boldrini}
\email{paolob@chalmers.se}
\address{ Department of Mathematical Sciences, Chalmers University of
	Technology and University of Gothenburg, Göteborg SE-412 96, Sweden
}

\begin{document}
	
	\begin{abstract} 
    We show that, for a second-countable locally compact Hausdorff étale minimal groupoid with compact unit space, simplicity of the reduced groupoid C*-algebra implies the existence of a comeager set of unit points with C*-simple isotropy group. Combining this result with work of Christensen and Neshveyev on exotic completions of isotropy group algebras, we show that the converse implication is also true. Finally, we construct a Hausdorff étale minimal groupoid with an isotropy group whose induced exotic completion differs from its reduced group C*-algebra, answering a question of Christensen and Neshveyev.
	\end{abstract}
	
	\maketitle

    \section{Introduction}

    Many $\cs$-algebras of interest are constructed from geometric or dynamical data. The framework of locally compact étale groupoids unifies many of these constructions, including reduced group $\cs$-algebras, reduced crossed products by discrete group actions, and graph $\cs$-algebras. A central problem is to determine how the structure of the resulting algebra reflects the properties of the underlying groupoid. In this paper we focus on simplicity, that is, the absence of non-trivial closed ideals.
    
     For actions of countable discrete groups, minimality of the action is a necessary condition for simplicity of $C(X)\rtimes_r G$, and for topologically amenable actions, simplicity of the crossed product is equivalent to the combination of minimality and topological freeness \cites{KawamuraTomiyama,ArchboldSpielberg}. This characterization fails outside the amenable world as shown by the trivial action of $\mathbb F_2$ on a single point. Powers showed that the associated $\cs$-algebra $\cs_r(\mathbb{F}_2)$ is simple \cite{Powers}, giving rise to the study of $\cs$-simple groups, that is, groups whose reduced $\cs$-algebra is simple.

    Given a locally compact Hausdorff étale groupoid $\cG$, a groupoid $\cs$-algebra is a completion of the convolution algebra of compactly supported continuous functions $C_c(\cG)$ with respect to a suitable $\cs$-norm. Such completions in general do not coincide, and in this paper we will mainly focus on the so-called reduced groupoid $\cs$-algebra $\cs_r(\cG)$, obtained from the regular representation of $\cG$. Simplicity of $\cs_r(\cG)$ has been studied extensively and connected to $\cs$-inclusions \cite{CrytserNagy}, exotic completions of isotropy group algebras \cites{ChristensenNeshveyev1,ChristensenNeshveyev2}, and amenable sections of isotropy groups \cite{KKLRU}.

    Let $\cG$ be a locally compact Hausdorff étale groupoid with compact unit space $X=\cG^{(0)}$. Given a point $x\in X$, its isotropy group $\cG_x^x$ is the group of arrows with source and range $x$. The main question investigated in this paper is the following:
    \begin{center}
        \textit{Is it possible to characterize simplicity of $\cs_r(\cG)$ in terms of $\cs$-simplicity of the isotropy groups of $\cG$?}
    \end{center}

    The first instance of this phenomenon was observed for discrete groups acting on compact spaces. In this context, the isotropy group at a point $x$ coincides with the stabilizer subgroup $G_x\defeq \{g\in G\colon gx=x\}$, and the algebra $\cs_r(\cG)$ coincides with the reduced crossed product $C(X)\rtimes_rG$. Ozawa showed \cite[Theorem 14]{Ozawa} that, when the action is minimal, $\cs$-simplicity of a single stabilizer forces simplicity of $C(X)\rtimes_rG$. Recently, Hartman and Kalantar showed that when a minimal action of a countable group gives rise to a simple crossed product, then there is a point whose stabilizer subgroup has trivial amenable radical \cite{HartmanKalantar}. Even more recently, Bray and Kennedy \cite{Bray-Kennedy} and, independently, the author \cite{Boldr} showed that when the acting group is countable the converse of Ozawa's result holds: simplicity of the reduced crossed product guarantees the existence of a point---or, equivalently, a comeager set of points---with $\cs$-simple stabilizer. Countability of the acting group cannot be removed, as shown in \cite[Example 3.6]{Bray-Kennedy}.

    We show that a similar equivalence holds in the more general setting of reduced groupoid $\cs$-algebras.

    \begin{introthm}\label{thm:mainshort}
        Let $\cG$ be a second-countable locally compact Hausdorff étale minimal groupoid with compact unit space $X$. Then the following are equivalent:
        \begin{enumerate}
            \item The reduced groupoid $\cs$-algebra $\cs_r(\cG)$ is simple;
            \item the set
            \[
            \{x\in X\colon \cG_x^x \text{ is $\cs$-simple}\} 
            \] is comeager in $X$.
        \end{enumerate}
    \end{introthm}

    Two main differences emerge when this statement is compared to the group-action counterpart. 
    
    The first is the assumption that $\cG$ is second-countable, which forces metrizability of the unit space $X$. The reason why this assumption is needed is that the proof of (1)$\Rightarrow$(2) is descriptive-set-theoretic in nature. In the group-action setting, a large part of this machinery can be avoided. The key difference is that the set of amenable subgroups of the acting group is closed in the Chabauty topology, while the subset of amenable isotropy subgroups is Borel, but not closed in the groupoid-Chabauty topology (see Proposition \ref{prop:amborel} and Proposition \ref{prop:amenability-not-closed}). The higher descriptive complexity in the groupoid setting requires more sophisticated techniques that are only available for Polish spaces.

    The second difference is that, unlike for transformation groupoids, it is not known whether the existence of a single unit point with $\cs$-simple isotropy is sufficient to guarantee simplicity of the reduced groupoid $\cs$-algebra. This is closely connected with the exotic completion of isotropy group algebras introduced by Christensen and Neshveyev \cite{ChristensenNeshveyev1}. Given a point $x\in X$ and an element $a\in \C[\cG_x^x]$, define the norm
    \[
    \|a\|_{e,x}\defeq \inf \{\|f\|_r\colon f\in C_c(\cG), f|_{\cG_x^x}=a\}.
    \]
    We will often suppress the subscript $x$ when the unit point is clear. The completion of $\C[\cG_x^x]$ with respect to $\|\cdot\|_e$ is denoted by $\cs_e(\cG_x^x)$, and in general does not coincide with the reduced group $\cs$-algebra $\cs_r(\cG_x^x)$.
    
    This exotic completion plays a crucial role in the proof of Theorem \ref{thm:mainshort} and allows us to extend the equivalence in Theorem \ref{thm:mainshort} to the following statement.

    \begin{introthm}\label{thm:mainlong}
       Let $\cG$ be a second-countable locally compact Hausdorff minimal étale groupoid with compact unit space. The following are equivalent:
       \begin{enumerate}
           \item the reduced groupoid $\cs$-algebra $\cs_r(\mathcal{G})$ is simple;
           \item there is a comeager set of unit points with $\cs$-simple isotropy;
           \item there is a comeager set of unit points with $\cs_e(\cG_x^x)$ simple;
           \item there is a point $x\in X$ with $\cs_e(\cG_x^x)$ simple;
           \item there is a point $x\in X$ such that $\cG_x^x$ is $\cs$-simple and $\cs_e(\cG_x^x)=\cs_r(\cG_x^x)$.
       \end{enumerate}
   \end{introthm}

   It is not clear whether the existence of a single unit point with $\cs$-simple isotropy is enough to force simplicity of $\cs_r(\cG)$.

   Most of the implications in the theorem are clear or follow from the work of Christensen and Neshveyev. The main contribution of the paper is the implication (1)$\Rightarrow$(2). Its proof is based on a Baire category argument, Kennedy's characterization of $\cs$-simple groups \cite{Kennedy}, and the characterization of simplicity of reduced groupoid $\cs$-algebras in terms of amenable confined sections of isotropy groups obtained by Kennedy, Kim, Li, Raum and Ursu \cite{KKLRU}. 

   Christensen and Neshveyev showed that for transformation groupoids associated with partial actions of discrete groups the exotic norm $\|\cdot\|_e$ coincides with the reduced norm $\|\cdot\|_r$ on every isotropy group algebra \cite[Corollary 4.15]{ChristensenNeshveyev1}. Combined with Theorem \ref{thm:mainlong} we obtain the following generalization of the results of Ozawa and Bray--Kennedy to partial actions.

   \begin{introcor}
       Let $G$ be a countable discrete group and $\alpha$ be a minimal partial action on a compact metrizable space $X$. The following are equivalent:
       \begin{enumerate}
           \item the reduced crossed product $C(X)\rtimes_{\alpha,r} G$ is simple;
           \item there is a comeager set of points in $X$ with $\cs$-simple stabilizer;
           \item there is a point $x\in X$ with $\cs$-simple stabilizer.
       \end{enumerate}
   \end{introcor}

   We conclude the paper with an example of a second-countable locally compact Hausdorff étale minimal groupoid $\cG$ with compact unit space such that
\begin{enumerate}
    \item the set of amenable isotropy subgroups is not closed in $\Sub(\cG)$;
    \item\label{property:introexoticity} there exists $x\in\cG^{(0)}$ such that
    \[
        \cs_e(\cG_x^x)\neq \cs_r(\cG_x^x).
    \]
\end{enumerate}

In particular, \eqref{property:introexoticity} answers \cite[Question 4.5]{ChristensenNeshveyev2}.

\noindent\textit{Acknowledgments.} The author would like to thank Spyros Petrakos and Jamie     Bell for  useful discussions. The author was supported by the Kungl. Vetenskapsakademien (MA2025-0041).\medbreak
	
	\noindent\textit{AI statement.} During the preparation of this work, the author used LLMs to assist with exposition, literature searches and to develop some of the arguments. The author independently reviewed the mathematical content generated this way.


\section{Notation and preliminaries}\label{sec:prelim}
    \subsection{Groupoids and their C*-algebras}
    We recall here some notation. For a complete introduction to groupoids and groupoid $\cs$-algebras we refer the reader to \cites{Renault,Sims}.
    
    A \emph{topological groupoid} $\cG$ is a groupoid, that is a small category where all morphisms are invertible, equipped with a locally compact topology that makes composition and inversion continuous maps, and such that the \emph{unit space} $\cG^{(0)}$ is Hausdorff. We always assume that a groupoid has nonempty unit space. We define the \emph{source} and \emph{range} maps by $s(g)=g\inv g$ and  $r(g)\defeq gg\inv$. An \emph{étale} groupoid is a topological groupoid such that $s$ and $r$ are local homeomorphisms. In this case, the topology of $\cG$ has a basis consisting of \emph{open bisections}, that is, open subsets $W\subseteq \cG$ such that the restrictions to $W$ of $s$ and $r$ are homeomorphisms.

    For $x,y\in X=\cG^{(0)}$, write
    \[ 
    \cG_x=s^{-1}(x),\qquad \cG^y=r^{-1}(y),\qquad \cG_x^y=\cG_x\cap\cG^y.
    \]
    The set $\cG_x^x$ is a group under composition, and it is called the \emph{isotropy group} at $x$.

    An étale groupoid is \emph{minimal} if for every unit point $x\in \cG^{(0)}$ the set $r(s\inv(x))$ is dense in $\cG^{(0)}$. 
    
    For an open bisection $W$, denote by $\theta_W$ the partial homeomorphism from $s(W)$ to $r(W)$ sending each point $x\in s(W)$ to the range of the unique element $g\in W$ with $s(g)=x$. 

    \begin{nota}
        Throughout the paper, $\cG$ will denote a second-countable locally compact Hausdorff minimal étale groupoid with compact unit space $X\defeq \cG^{(0)}$. 
    \end{nota}

    Let $C_c(\cG)$ denote the vector space of continuous compactly supported complex-valued functions. For $f_1,f_2\in C_c(\cG)$, set
        \[
    (f_1*f_2)(g)
      =\sum_{h\in\cG^{r(g)}}f_1(h)f_2(h^{-1}g),
    \]
    and $f^*(g)=\overline{f(g^{-1})}.$
    For $x\in X$, the \emph{regular representation} $\lambda_x\colon C_c(\cG)\to B(\ell^2(\cG_x))$
    is given by $(\lambda_x(f)\xi)(g)=\sum_{h\in\cG_x}f(gh^{-1})\xi(h).$

    The \emph{reduced norm} is
    \[
    \|f\|_r=\sup_{x\in X}\|\lambda_x(f)\|,
    \]
    and $\cs_r(\cG)$ is the completion of $C_c(\cG)$ with respect to this norm.

    In \cites{ChristensenNeshveyev1, ChristensenNeshveyev2}, Christensen and Neshveyev associated possibly exotic $\cs$-algebras to isotropy groups of a (not necessarily Hausdorff) locally compact étale groupoid. We recall some of the relevant results obtained in their papers, but for simplicity we stick to the Hausdorff and second-countable setting. We also restrict to the reduced groupoid $\cs$-algebra, while most of their results are valid for more general completions.

    The following is the key notion introduced by Christensen and Neshveyev. We give here a simplified form, the equivalence with the original definition is \cite[Theorem 1.8]{ChristensenNeshveyev2}.

    \begin{defn}
        Let $x\in X$. For an element $a\in \C [\cG_x^x]$, let
        \[
        \|a\|_{e,x}=\inf \{\|f\|_r\colon f\in C_c(\cG), f|_{\cG_x^x}=a\}.
        \]
    \end{defn}

    We often omit the unit point $x$ from the notation in the norm. We denote by $\cs_e(\cG_x^x)$ the completion of $\C [\cG_x^x]$ with respect to $\|\cdot\|_{e}$ and we refer to it as the \emph{exotic completion} of $\cG_x^x$.
    
    Since $\|\cdot\|_e\geq\|\cdot\|_r$, the canonical inclusion $\C[\cG_x^x]\hookrightarrow \cs_r(\cG_x^x)$ extends to a quotient homomorphism $q_x\colon \cs_e(\cG_x^x)\to \cs_r(\cG_x^x)$. We write $\cs_e(\cG_x^x)= \cs_r(\cG_x^x)$ when the map $q_x$ is injective, hence an isomorphism.

    Since the groupoids we consider are Hausdorff, their singular ideal is zero, so \cite[Theorem 5.1]{ChristensenNeshveyev2} shows that the equality $\cs_e(\cG_x^x)= \cs_r(\cG_x^x)$ occurs quite often:

    \begin{thm}
        The set 
        \[
        \{x\in X\colon \cs_e(\cG_x^x)= \cs_r(\cG_x^x)\}
        \]
        is comeager in $X$.
    \end{thm}

    We end the subsection with a criterion for simplicity of reduced groupoid $\cs$-algebras \cite[Corollary 4.4]{ChristensenNeshveyev2}. The original statement is more general, but reduces to the following if the groupoid is assumed to be Hausdorff.

    \begin{thm}
        Suppose there is a point $x\in X$ such that $\cG_x^x$ is $\cs$-simple and $\cs_e(\cG_x^x)=\cs_r(\cG_x^x)$. Then $\cs_r(\cG)$ is simple.
    \end{thm}
    
    \subsection{The groupoid-Chabauty space and simplicity criteria}
    Given a countable discrete group $G$, the \emph{Chabauty space} $\Sub(G)$ is the closed subset of $\{0,1\}^G$ consisting of characteristic functions of subgroups of $G$. We equip $\{0,1\}^G$ with the product topology and $\Sub(G)$ with the subspace topology. This makes $\Sub(G)$ a compact metrizable space, called the Chabauty space of the group $G$.
   
    A natural generalization of the Chabauty space to the context of locally compact étale groupoids has been introduced and studied in \cite[Section 7]{KKLRU}. We recall some terminology here, specializing it to the Hausdorff and second countable setting. Denote by $\mathfrak C(\cG)$ the space of closed subsets of $\cG$ (including the empty set) equipped with the \emph{Fell topology}. A subbasis for such a topology is given by the touch-open and miss-compact sets
    \[
        T_U\defeq\{F\in \mathfrak C(\cG)\colon F\cap U\neq\emptyset \}, \qquad M_K\defeq\{F\in \mathfrak C(\cG)\colon F\cap K=\emptyset \},
    \]
    where $U\subseteq \cG$ is open and $K\subseteq \cG$ is compact. The Fell topology turns $\mathfrak C(\cG)$ into a compact metrizable space \cite[Exercise 12.7]{Kechris}. 
    
    Continuity of the range and the source map makes each isotropy group $\cG_x^x$ closed. Moreover, the étale condition combined with second countability shows that each isotropy group is countable and discrete, hence whenever $H\leq \cG_x^x$, the group $H$ is closed in $\cG$, hence is an element of $\mathfrak C(\cG)$. The space $\Sub(\cG)$ of subgroups of isotropy groups is closed \cite[Lemma 7.4]{KKLRU} in the Fell topology, hence it is compact metrizable when equipped with the \emph{groupoid-Chabauty topology}, that is, the restriction of the Fell topology to $\Sub(\cG)$. More concretely, a base for the topology is again given by the touch-open and miss-compact sets
     \begin{equation}\label{eq:touch-open-miss-compact}
         T_U\defeq\{F\in \Sub(\cG)\colon F\cap U\neq\emptyset \}, \qquad M_K\defeq\{F\in \Sub(\cG)\colon F\cap K=\emptyset \},
     \end{equation}
     where $U\subseteq \cG$ is a precompact open bisection and $K\subseteq \cG$ is compact.
     
    It is readily verified that the map $p\colon \Sub(\cG)\to X$ that associates to each subgroup of an isotropy group its unit point is continuous. We sometimes write $H_x$ if we want to stress that $H\leq \cG_x^x$. 

    The groupoid $\cG$ acts on $\Sub(\cG)$ as follows: if $x\in X$, $H\leq \cG_x^x$, and $g\in \cG_x$, then $gHg\inv$ is a subgroup of $\cG_y^y$, where $y=r(g)$. We set $g\cdot H=gHg\inv$.

   We recall two characterizations crucially used in the proof of Theorem \ref{thm:main}. The first is Kennedy's characterization of $\cs$-simplicity \cite[Theorem~1.1]{Kennedy}. 
    
    \begin{thm}\label{thm:kennedy}
    Let $H$ be a discrete group.
    Then $H$ is not $\cs$-simple if and only if there exist an amenable subgroup $K\leq H$ and a finite set $F\subseteq H\setminus\{e\}$ such that
    \[
     hKh\inv\cap F\neq\varnothing
     \qquad\text{for every }h\in H.
    \]
    \end{thm}

    Building on a sufficient criterion of Borys \cite[Corollary 6.6]{Borys}, Kennedy, Kim, Li, Raum and Ursu characterized simplicity of the essential groupoid $C^*$-algebra in terms of the absence of essentially confined amenable sections \cite[Theorem~7.14]{KKLRU}. 
    
    A \emph{section of isotropy groups} is a family $\Lambda=\{H_x\colon x\in \mathfrak X\}$ with $H_x\leq \cG_x^x$ and $\mathfrak X\subseteq X$ dense. The section $\Lambda$ is \emph{amenable} if every $H_x$ is an amenable group, and it is \emph{confined} if there is a point $x\in X$ such that every $H\in \overline{\cG\cdot\Lambda}$ with $p(H)=x$ is not contained in the unit space. That is, setting $Z_x(\Lambda)\defeq\{H\in \overline{\cG\cdot\Lambda}\colon p(H)=x\}$, the section $\Lambda$ is confined if and only if 
    \begin{equation} \label{eq:KKLRU}
        \exists x\in X \forall H\in Z_x(\Lambda),\quad H\not\subseteq X.
    \end{equation}

    We point out that in \cite[Definition 7.1]{KKLRU}, the authors define the notion of \emph{essentially} confined section. The difference is due to the fact that non-Hausdorff groupoids have non-closed unit space. Since we are assuming Hausdorffness throughout the paper, we drop the word ``essentially'' from the definition. Since the essential and reduced groupoid $C^*$-algebras agree for Hausdorff groupoids, their result takes the following form in our setting.
    \begin{thm}\label{thm:KKLRU}
        Let $\cG$ be a second-countable locally compact Hausdorff étale minimal groupoid with compact unit space. Then $\cs_r(\cG)$ is simple if and only if there is no amenable confined section of isotropy groups.
    \end{thm}
    
   \subsection{The Furstenberg URS}
    Let $H$ be a countable group and let $\partial_F H$ denote its Furstenberg boundary. The \emph{Furstenberg URS} is defined as the set
    \[
    \cA_H=\{H_z\colon z\in\partial_F H\}.
    \]
    
    Le Boudec and Matte Bon showed that every member of $\cA_H$ is amenable, and the conjugation action $H\acts\cA_H$ is a boundary action, see \cite[Theorem~2.16 and Proposition~2.21]{LeBoudecMatteBon}. Moreover, they showed that the Furstenberg URS can be characterized by a universal property. As a consequence, every isomorphism of groups $\varphi\colon H\to H'$ induces a homeomorphism 
    \[
    \varphi_*\colon\Sub(H)\to\Sub(H'),\qquad L\longmapsto\varphi(L),
    \]
    which satisfies $\varphi_*(\cA_H)=\cA_{H'}$.

    \begin{lem}\label{lem:urs-obstruction}
    Let $H$ be a countable discrete group, and let $F\subseteq H\setminus\{e\}$ be finite. Suppose there is an amenable subgroup $K\leq H$ such that
    \[
    hKh\inv \cap F\neq\emptyset \qquad\text{for all $h\in H$}.    
    \]
    Then 
    \[
    A\cap F\neq\varnothing
    \qquad\text{for every }A\in\cA_H.
    \]
    \end{lem}

    \begin{proof}
         Denote by $Q_F$ the set of subgroups of $H$ that intersect $F$. Then $Q_F$ is closed in the Chabauty topology and $Q\defeq\overline{\{gKg\inv\colon g\in H\}}\subseteq Q_F$. The set $Q$ is closed, $H$-invariant, and all its elements are amenable subgroups of $H$, as amenability is a closed condition in $\Sub(H)$. Let $K_0\in Q$. Then $K_0\acts \partial_FH$ and is amenable, therefore there is a $K_0$-invariant probability measure $\mu$ on  $\partial_FH$. By strong proximality there is a net $(g_i)_i\subseteq H$ and an element $z\in \partial_FH$ such that $g_i\mu\to\delta_z$. Consider the net $g_iK_0g_i\inv$ in $Q$. By compactness, up to passing to a subnet we can assume that $g_iK_0g_i\inv\to K_1\in Q$.

        We claim that $K_1\leq H_z$. Indeed, by definition of the Chabauty topology if $k\in K_1$, then $k\in g_iK_0g_i\inv$ eventually, and therefore $k(g_i\mu)=g_i\mu$. Since the left hand side converges to $\delta_{kz}$ while the right hand side converges to $\delta_z$, it follows that $kz=z$.
        
        Let $A\in \mathcal{A}_H$. Since the action $H\acts \mathcal{A}_H$ is minimal and $H_z\in \cA_H$, there is a net $(h_i)\subseteq H$ such that $h_iH_zh_i\inv\to A$. Since $(h_iK_1h_i\inv)$ is a net in $Q$, up to passing to a subnet we can assume that $h_iK_1h_i\inv\to K_2$ for some $K_2\in Q$. Since the subgroup relation is closed, $K_2\leq A$. Moreover $K_2\cap F\neq \emptyset$ since $K_2\in Q_F$, hence $A\cap F\neq\emptyset$.
    \end{proof}

   \subsection{Descriptive set theory}
   Descriptive set theory studies the topological complexity of subsets of Polish spaces, that is, completely metrizable and separable spaces. 

   \begin{defn}
       Let $Y$ be a Polish space. The \emph{Borel $\sigma$-algebra} is the smallest $\sigma$-algebra containing all open subsets of $Y$. A \emph{Borel set} is an element of the Borel $\sigma$-algebra. A map between Polish spaces $f\colon Z\to Y$ is \emph{Borel} if the inverse image of any Borel set is Borel. A subset $A$ of $Y$ is \emph{analytic} if there exist a Polish space $Z$, a Borel subset $B\subseteq Z$, and a Borel map $f\colon Z\to Y$ such that $A=f(B)$.
   \end{defn}
   
   A common theme in descriptive set theory is that sets of low complexity have some interesting structural properties. The relevant property for the current paper is the \emph{Baire property}.

   \begin{defn}
       Let $Y$ be a Polish space. A subset of $Y$ is \emph{meager} if it is the union of sets whose closure has empty interior. A set is \emph{comeager} if its complement is meager. A subset of $Y$ has the \emph{Baire property} if it can be written as a symmetric difference of an open and a meager set.
   \end{defn}

   If $B\subseteq Y$ has the Baire property, then $B=U\Delta M$ with $U$ open and $M$ meager. If in addition $B$ is not meager, then the open set $U$ is nonempty. Moreover, $B$ is comeager in $U$, that is, $B\cap U$ is comeager as a subset of $U$. 
   
   It is a standard fact that analytic subsets of Polish spaces have the Baire property, see \cite[Theorem 21.6]{Kechris}. It is therefore desirable to have tools to control complexity of subsets. The most relevant here will be the Lusin-Novikov Theorem \cite[Theorem 18.10]{Kechris}.  

   \begin{thm}
       Let $X$ and $Y$ be Polish spaces, and $R\subseteq X\times Y$ be Borel. If all the fibers $R_x=\{y\in Y\colon (x,y)\in R\}$ are countable, then there are a sequence $(A_n)_{n\in \N}$ of Borel subsets of $X$ and a sequence of Borel maps $f_n\colon A_n\to Y$ such that $(x,y)\in R$ if and only if there is $n\in \N$ such that $x\in A_n$ and $y=f_n(x)$.
   \end{thm}

   As an easy consequence, the image of a Borel set under a countable-to-one Borel map is Borel \cite[Exercise 18.14]{Kechris}.

\section{Descriptive complexity of useful subsets}
    For a compact subset $F\subseteq \cG\setminus X$, set 
    \begin{equation}\label{eq:BF}
    B_F=\left\{x\in X\colon
    \begin{array}{l}
    \text{there exists an amenable }K\leq\cG_x^x \text{ such that}\\[2pt]
     hKh\inv\cap F\neq\emptyset\text{ for every }h\in \cG_x^x
    \end{array}
    \right\}.
    \end{equation}
    Kennedy's criterion \ref{thm:kennedy} will allow us to write the set of points with non $\cs$-simple isotropy as a subset of a countable union of sets of the form $B_F$. The main goal of this section is to show that each $B_F$ is analytic in $X$, hence it has the Baire property.

    Consider the \emph{isotropy map} $\Iso\colon X\to \Sub(\cG)$ that sends $x$ to its isotropy group $\cG_x^x$. Then $B_F=\Iso\inv(C_F)$, where 
    \begin{equation}\label{eq:CF}
    C_F=\left\{H\in \Sub(\cG)\colon
    \begin{array}{l}
    \text{there exists an amenable }K\leq H \text{ such that}\\[2pt]
     hKh\inv\cap F\neq\emptyset\text{ for every }h\in H
    \end{array}
    \right\}.
    \end{equation}

    \begin{lem}\label{lem:Isoborel}
        The map $\Iso\colon X\to \Sub(\cG)$ is Borel.
    \end{lem}
   \begin{proof}
       It suffices to check that the inverse image of the generators of the groupoid-Chabauty topology described in \eqref{eq:touch-open-miss-compact} is Borel. Notice that, for an open set $U\subseteq \cG$,
       \[
       \Iso\inv(T_U)=s(\Iso(\cG)\cap U).
       \]
       Continuity of the source and range maps implies that $\Iso(\cG)$ is closed, hence $\Iso(\cG)\cap U$ is Borel. Since $\cG$ is étale and second-countable, the map $s$ is countable-to-one, hence the Lusin-Novikov theorem implies that the image $s(\Iso(\cG)\cap U)$ is Borel. Similarly $\Iso\inv(\Sub(\cG)\setminus M_L)=s(\Iso(\cG)\cap L)$ is closed for compact $L\subseteq \cG$. Hence $\Iso\inv( M_L)=X\setminus s(\Iso(\cG)\cap L)$ is open, showing that $\Iso$ is Borel.
   \end{proof}

    \begin{prop}\label{prop:amborel}
        The set $\Am(\cG)\defeq \{K\in \Sub(\cG)\colon K \text{ is amenable}\}$ is Borel in $\Sub(\cG)$.
    \end{prop}
    \begin{proof}
        Let $R\subseteq \Sub(\cG)\times \cG$ be the set $\{(H,g)\colon g\in H\}$. It is readily checked that $R$ is closed, and since every $H\in \Sub(\cG)$ is countable, all the fibers $R_H=\{h\colon (H,h)\in R\}$ are countable. 
        
        By the Lusin-Novikov theorem, there are Borel subsets $A_n\subseteq \Sub(\cG)$ and Borel maps $f_n\colon A_n\to \cG$ such that $(H,g)\in R$ if and only if there is $n\in \N$ with $H\in A_n$ and $g=f_n(H)$. Each $f_n$ can be extended to a total Borel map $e_n\colon \Sub(\cG)\to \cG$ given by 
        \[
        e_n(H)=\begin{cases}
            f_n(H) \quad \text{if }H\in A_n \\
            p(H)  \quad \text{if }H\notin A_n
        \end{cases},
        \]
        where $p$ is the continuous map associating the unit point to each isotropy subgroup. 
        
        Then for all $H\in \Sub(\cG)$, one has the Borel enumeration $H=\{e_n(H)\colon n\in \N\}$. The group $H$ is amenable if and only if for all $m\in \N_{\geq 1}$ and all nonempty finite subsets $I\subseteq H$ there is a nonempty finite subset $J\subseteq H$ such that for all $g\in I$
        \[
        |gJ\Delta J|< |J|/m.
        \]
        
        Every vector $\mathbf{i}\in \N^p$ induces a finite subset $S_{\mathbf{i}}(H)=\{e_{i_l}(H)\colon l=1,\dots, p\}$. Conversely, for each nonempty finite subset $I\subseteq H$ there is $p\in \N$ and $\mathbf{i}\in \N^p$ such that $I=S_{\mathbf{i}}(H)$. Therefore
        \[
        \Am(\cG)=\bigcap_{m\in \N_{\geq 1}}\bigcap_{p\in \N_{\geq 1}}\bigcap_{\mathbf{i}\in \N^p}\bigcup_{q\in \N_{\geq 1}}\bigcup_{\mathbf{j}\in \N^q} A_{m,p,\mathbf{i},q,\mathbf{j}},
        \]
        where $A_{m,p,\mathbf{i},q,\mathbf{j}}$ is the set of $H\in \Sub(\cG)$ such that 
        \begin{enumerate}
            \item \label{cond:distinct} the elements $e_{j_k}(H)$ are pairwise distinct,
            \item \label{cond:Folner} $|e_{i_l}S_{\mathbf{j}}(H)\Delta S_{\mathbf{j}}(H)|<q/m$ for all $l=1,\dots, p$.
        \end{enumerate}

        Condition \eqref{cond:distinct} can be written as a conjunction of formulas of the type $e_{j_l}(H)\neq e_{j_k}(H)$. Since the maps $e_n$ are Borel, these conditions are all Borel, hence condition \eqref{cond:distinct} is Borel. Condition \eqref{cond:Folner} can be written as a finite disjunction of cardinality equalities, and each cardinality equality can be expressed as a Boolean combination of conditions of the form $e_{i_l}(H) e_{j_k}(H)=e_{j_r}(H)$ or $e_{i_l}(H) e_{j_k}(H)\neq e_{j_r}(H)$. These conditions are all Borel since multiplication is continuous on $\cG^{(2)}$, hence condition \eqref{cond:Folner} is Borel. Therefore, $\Am(\cG)$ is Borel.       
    \end{proof}

    In Section \ref{sec:counterexample}, we will show that $\Am(\cG)$ is not closed in general.

    \begin{thm}\label{thm:BFCFanalytic}
        Let $F\subseteq \cG\setminus X$ be compact. Then the set $B_F$ defined in \eqref{eq:BF} and the set $C_F$ defined in \eqref{eq:CF} are analytic.
    \end{thm}
    \begin{proof}
        Since the map $\Iso$ is Borel by Lemma \ref{lem:Isoborel} and $B_F=\Iso\inv(C_F)$, it suffices to show that $C_F$ is analytic.

        Consider the set $P_F\subseteq \Sub(\cG)^2$ consisting of all pairs $(H,K)$ satisfying
        \begin{enumerate}
            \item \label{cond:amenability}$K$ is amenable;
            \item \label{cond:inclusionconjugacy} $K\leq H$ and $hKh\inv\cap F\neq \emptyset$ for all $h\in H$.
        \end{enumerate}

        Condition \eqref{cond:amenability} is Borel by Proposition \ref{prop:amborel}. 
        
        Suppose $(H,K)\in \Sub(\cG)^2$ and $K\not\leq H$. Then there is an element $k\in K\setminus H$. Since $H$ is closed, there is a precompact open bisection $W$ containing $k$ and such that $H\cap \overline{W}=\emptyset$. Therefore $(H,K)\in M_{\overline{W}}\times T_W$, and for all $(H',K')\in M_{\overline{W}}\times T_W$, we have $K'\not\leq H'$, hence the relation $K\leq H$ is closed. 
        
        Suppose now $(H_n, K_n)$ is a sequence in $\Sub(\cG)^2$ satisfying condition \eqref{cond:inclusionconjugacy} and converging to $(H,K)$ in the groupoid-Chabauty topology. By the previous observation $K\leq H$. Let $h\in H$ and let $(U_m)_{m\in\N}$ be a decreasing local basis of open bisections containing $h$. For every $m\in \N$, eventually $H_n\in T_{U_m}$, hence there is $n_m\in \N$ and $h_m\in H_{n_m}\cap U_m$. Therefore, up to passing to a subsequence we can find a sequence of elements $h_n\in H_n$ such that $h_n\to h$ in the groupoid topology. Since $K_n\cap h_n\inv Fh_n\neq \emptyset$, there is a sequence $f_n$ of elements of $F$ such that $h_n\inv f_nh_n\in K_n$. By compactness of $F$ we can assume, again after passing to a subsequence, that $f_n\to f$ for some $f\in F$. We claim that $h\inv f h\in K$. Indeed, otherwise there would be a precompact open bisection $V$ containing $h\inv f h$ such that $\overline{V}\cap K=\emptyset$. Then $K_n\in M_{\overline{V}}$ eventually, hence $K_n\cap V=\emptyset$, which contradicts the fact that $h_n\inv f_n h_n\in K_n$ and $h_n\inv f_n h_n\in V$ eventually. This shows that $hKh\inv\cap F\neq \emptyset$, and hence, since $h\in H$ was arbitrary, that $(H,K)$ also satisfies condition \eqref{cond:inclusionconjugacy}.

        Conditions \eqref{cond:amenability} and \eqref{cond:inclusionconjugacy} are Borel, hence $P_F$ is Borel. Since $C_F$ is the projection of $P_F$ on the first component, the set $C_F$ is analytic, concluding the proof. 
    \end{proof}

\section{Proof of the main theorem}

    Let $B=\{x\in X\colon \cG_x^x \text{ is not $\cs$-simple} \}$. To show meagerness of $B$ we will show it is contained in a union of countably many sets $B_F$, and that each of those is meager.

    Let $W\subseteq\cG$ be a precompact open bisection and let $D\subseteq s(W)$ be compact. For $z\in s(W)$, let $w_z\in W$ be the unique arrow with source $z$. If $F\subseteq\cG\setminus X$ is compact, define
    \begin{equation}\label{eq:local-Ad}
    \Ad_{W,D}(F)=\{w_zfw_z^{-1}\colon z\in D,\ f\in F\cap\cG_z^z\}.
    \end{equation}
    
    \begin{lem}\label{lem:local-equivariance}
    The set $\Ad_{W,D}(F)$ is a compact subset of $\cG\setminus X$, and
    \[
    \theta_W(D\cap B_F)\subseteq B_{\Ad_{W,D}(F)}.
    \]
    \end{lem}
    
    \begin{proof}
    The map $z\mapsto w_z$ is continuous. The set of pairs $(z,f)\in D\times F$ satisfying $f\in\cG_z^z$ is closed and hence compact. Its image under $(z,f)\mapsto w_zfw_z^{-1}$ is therefore compact and coincides with $\Ad_{W,D}(F)$. Conjugation sends a unit to a unit and a nonunit isotropy arrow to a nonunit isotropy arrow, so the image is disjoint from $X$.
    
    Let $x\in D\cap B_F$, put $y=\theta_W(x)$, and let $K\leq\cG_x^x$ witness $x\in B_F$. Then $K'=w_xKw_x^{-1}$ is amenable. Every $h'\in\cG_y^y$ has the form $h'=w_xhw_x^{-1}$ for a unique $h\in\cG_x^x$. Choose $f\in hKh^{-1}\cap F$. It follows that
    \[
    w_xfw_x^{-1}\in h'K'(h')^{-1}\cap\Ad_{W,D}(F),
    \]
    so $K'$ witnesses $y\in B_{\Ad_{W,D}(F)}$.
    \end{proof}
 
    \begin{lem}\label{lem:comeager}
        Suppose there is a compact $F\subseteq \cG\setminus X$ such that $B_{F}$ is not meager. Then there is a compact $F_0\subseteq\cG\setminus X$ such that $B_{F_0}$ is comeager in $X$.
    \end{lem}
    \begin{proof}
        By Theorem \ref{thm:BFCFanalytic}, $B_{F}$ is analytic, hence it has the Baire property. 
        
        Let $U\subseteq X$ be a nonempty open set in which $B_{F}$ is comeager and let $y\in X$.  By minimality of $\cG$, there is a point $x\in U$ and an element $g\in \cG$ with $s(g)=x$ and $r(g)=y$. Let $W$ be a precompact bisection containing $g$, and let $D$ be an open neighborhood of $x$ whose closure is contained in $U\cap s(W)$. Then $y\in \theta_W(D)$. By compactness there are precompact open bisections $W_j$ and open sets $D_j\subseteq \overline{D_j}\subseteq U\cap s(W_j)$ such that $X=\bigcup_{j=1}^n \theta_{W_j}(D_j)$. 
        
        Fix $j\in\{1,\dots n\}$. Since $B_{F}\cap U$ is comeager in $U$, $B_{F}\cap D_j$ is comeager in $D_j$, hence $\theta_{W_j}(D_j\cap B_{F})$ is comeager in $\theta_{W_j}(D_j)$. By Lemma \ref{lem:local-equivariance}, we have that $\theta_{W_j}(D_j\cap B_{F})\subseteq B_{\Ad_{W_j,\overline{D_j}}(F)}$. Let $F_0=\bigcup_{j=1}^n\Ad_{W_j,\overline{D_j}}(F)$. Then $F_0\subseteq \cG\setminus X$ is compact and $B_{F_0}\supseteq B_{\Ad_{W_j,\overline{D_j}}(F)}\supseteq \theta_{W_j}(D_j\cap B_F)$ is comeager in each $\theta_{W_j}(D_j)$. 
        
        As these open sets form a finite open cover of $X$, the set $B_{F_0}$ is comeager in $X$.
    \end{proof}

    We will also need a groupoid version of Lemma \ref{lem:urs-obstruction}.

    \begin{lem}\label{lem:groupoid-urs-obstruction}
    Let $F\subseteq\cG\setminus X$ be compact and let $x\in B_F$. Then
    \[
    A\cap F\neq\varnothing
    \qquad\text{for every }A\in\cA_{\cG_x^x}.
    \]
    \end{lem}
    
    \begin{proof}
    Since $\cG_x^x$ is a closed discrete subspace of $\cG$, the set
    \[
    E=F\cap\cG_x^x
    \]
    is finite. Any amenable subgroup witnessing $x\in B_F$ satisfies the hypothesis of Lemma~\ref{lem:urs-obstruction} with this $E$. Since $E\subseteq F$, it follows that $A\cap F\neq \emptyset$.
    \end{proof}

    \begin{prop}\label{prop:BFmeager}
    Let $\cG$ be a second-countable locally compact Hausdorff étale minimal groupoid with compact unit space $X$. If $\cs_r(\cG)$ is simple, then for all compact $F\subseteq\cG\setminus X$ the set $B_F$ is meager.        
    \end{prop}

    \begin{proof}
        Suppose by contradiction that there is a compact $F$ such that $B_{F}$ is not meager. Then by Lemma \ref{lem:comeager}, there is a compact $F_0\subseteq \cG\setminus X$ such that $B_{F_0}$ is comeager in $X$. 
        
        Choose a countable cover $(V_n)_{n\geq1}$ of $\cG$ by open bisections, and put $M=X\setminus B_{F_0}$. Then
        \[
        \cG\cdot M=r(s^{-1}(M))
        =\bigcup_{n\geq1}\theta_{V_n}(M\cap s(V_n))
        \]
        is meager. Therefore
        \[
        Y=X\setminus(\cG\cdot M)
        \]
        is comeager, invariant, and contained in $B_{F_0}$.
        
   Fix $x\in Y$ and choose $A\in\cA_{\cG_x^x}$. Let $O=\cG\cdot x$. This set is dense in $X$ by minimality and satisfies $O\subseteq Y\subseteq B_{F_0}$. For every $y\in O$, choose an arrow $\gamma_y\colon x\to y$ and put
    \[
    A_y=\gamma_yA\gamma_y^{-1}\leq\cG_y^y.
    \]
    
    The family
    \[
    \Lambda=\{A_y\colon y\in O\}
    \]
    is an amenable section, since $A$ and all its translates are amenable.
    
    Let $L$ be an element of $\cG\cdot\Lambda$. Then there is an element $g\in \cG$ with $s(g)=x$ and $r(g)=z\in O$ such that $L=gAg\inv\leq \cG_z^z$. Then $L\in \cA_{\cG_z^z}$, and since $z\in B_{F_0}$, $L\cap F_0\neq\emptyset$.
 
    The set
    \[
    \{L\in\Sub(\cG):L\cap F_0\neq\varnothing\}
    \]
    is closed, because its complement is a miss-compact open set. Hence every element of $\overline{\cG\cdot\Lambda}$ meets $F_0$. Since $F_0\cap X=\varnothing$, no such subgroup is contained in $X$. Thus $\Lambda$ is a confined amenable section, contradicting Theorem~\ref{thm:KKLRU}.
    \end{proof}

    \begin{thm}\label{thm:main}
       Let $\cG$ be a second-countable locally compact Hausdorff étale minimal groupoid with compact unit space $X$. If $\cs_r(\cG)$ is simple, then the set of unit points with $\cs$-simple isotropy is comeager. 
    \end{thm}
    
    \begin{proof}
         The subspace $\cG\setminus X$ is locally compact Hausdorff and second-countable. Therefore it admits a countable increasing exhaustion by compact sets $F_n\subseteq \cG\setminus X$. 
         
         If $x\in B$, then by Kennedy's criterion \ref{thm:kennedy} there is a finite subset $F\subseteq \cG_x^x\setminus \{x\}$ and an amenable subgroup $K\leq \cG_x^x$ such that $hKh\inv\cap F\neq\emptyset$ for all $h\in \cG_x^x$. This shows that $x\in B_F$. Since $(F_n)_{n\in \N}$ is an exhaustion of $\cG\setminus X$, there is an index $n$ such that $F\subseteq F_n$, and therefore $B_F\subseteq B_{F_n}$. This proves $B\subseteq \bigcup_{n\in \N} B_{F_n}$. 
         
         By Proposition \ref{prop:BFmeager}, every set $B_{F_n}$ is meager, hence $B$ is meager as well.
    \end{proof}

   We are now ready to state and prove the main theorem.

   \begin{thm}\label{thm:mainproved}
       Let $\cG$ be a second-countable locally compact Hausdorff minimal étale groupoid with compact unit space. The following are equivalent:
       \begin{enumerate}
           \item the reduced groupoid $\cs$-algebra $\cs_r(\mathcal{G})$ is simple;
           \item there is a comeager set of unit points with $\cs$-simple isotropy;
           \item there is a comeager set of unit points with $\cs_e(\cG_x^x)$ simple;
           \item there is a point $x\in X$ with $\cs_e(\cG_x^x)$ simple;
           \item there is a point $x\in X$ such that  $\cG_x^x$ is $\cs$-simple and $\cs_e(\cG_x^x)=\cs_r(\cG_x^x)$.
       \end{enumerate}
   \end{thm}
   \begin{proof}
       The implication (1)$\Rightarrow$(2) is Theorem \ref{thm:main}. 
       
       By \cite[Theorem 5.1]{ChristensenNeshveyev2}, the set of points with $\cs_e(\cG_x^x)=\cs_r(\cG_x^x)$ is comeager, which implies (2)$\Rightarrow$(3). 
       
       The implication (3)$\Rightarrow$(4) is obvious, and (4)$\Rightarrow$(5) follows from the fact that the canonical quotient map $q_x\colon \cs_e(\cG_x^x)\to\cs_r(\cG_x^x)$ must be an isomorphism when $\cs_e(\cG_x^x)$ is simple. 
       
       Finally, the implication (5)$\Rightarrow$(1) is \cite[Corollary 4.4]{ChristensenNeshveyev2} applied to Hausdorff groupoids.
   \end{proof}

   In analogy with the crossed product case, it is natural to ask:

\begin{qst}\label{qst:onepoint}
    Let $\cG$ be a second-countable locally compact Hausdorff étale minimal groupoid with compact unit space, and suppose there is a unit point $x$ such that the isotropy group $\cG_x^x$ is $\cs$-simple. Does it follow that the reduced groupoid $\cs$-algebra is simple?
\end{qst}

An affirmative answer to Question \ref{qst:onepoint} would follow from a negative answer to \cite[Question 4.5]{ChristensenNeshveyev2}. In the next section, we present a groupoid that answers \cite[Question 4.5]{ChristensenNeshveyev2} positively. Nevertheless, we show that the associated reduced $\cs$-algebra is simple, leaving Question \ref{qst:onepoint} open.

Christensen and Neshveyev proved several criteria for the absence of exotic completions in graded groupoids, see \cite[Section 4]{ChristensenNeshveyev1}. In particular, they showed that for every unit point $x$ of an étale groupoid associated with a partial action of a discrete group on a locally compact Hausdorff space, one has $\|\cdot\|_{e,x}=\|\cdot\|_{r,x}$. Combined with Theorem \ref{thm:mainproved}, we obtain the following generalization of the results of Ozawa and Bray--Kennedy to partial actions on compact metrizable spaces.

\begin{cor}
    Let $G$ be a countable discrete group and $\alpha$ be a minimal partial action on a compact metrizable space $X$. The following are equivalent:
       \begin{enumerate}
           \item\label{cond:simplicity} the reduced crossed product $C(X)\rtimes_{\alpha,r} G$ is simple;
           \item\label{cond:comeager} there is a comeager set of points in $X$ with $\cs$-simple stabilizer;
           \item\label{cond:onepoint} there is a point $x\in X$ with $\cs$-simple stabilizer.
       \end{enumerate}
\end{cor}

\begin{proof}
    Let $\cG_\alpha$ be the transformation-type groupoid associated to $\alpha$, as in \cite[Example 4.2]{ChristensenNeshveyev1}.

    The stabilizer of a point $x$ under the partial action $\alpha$ is canonically identified with $(\cG_\alpha)_x^x$ and by \cite[Proposition 2.2]{Li}, the crossed product $C(X)\rtimes_{\alpha,r} G$ is canonically isomorphic to the reduced $\cs$-algebra of $\cG_\alpha$. This groupoid satisfies the hypotheses of Theorem \ref{thm:main}, hence the implication \eqref{cond:simplicity}$\Rightarrow$\eqref{cond:comeager} follows. The implication \eqref{cond:comeager}$\Rightarrow$\eqref{cond:onepoint} is immediate. 
    
    Suppose that $(\cG_\alpha)_x^x$ is $\cs$-simple for some $x\in X$. By \cite[Corollary 4.15]{ChristensenNeshveyev1}, the exotic completion $\cs_e((\cG_\alpha)_x^x)$ coincides with $\cs_r((\cG_\alpha)_x^x)$ and is therefore simple. Theorem \ref{thm:mainproved} and the canonical isomorphism above imply that $C(X)\rtimes_{\alpha,r} G$ is simple.
\end{proof}

\section{Free products of groupoids and a minimal groupoid with exotic isotropy}\label{sec:counterexample}

In this section, we recall the definition of free products of groupoids, and we equip it with a topology. We use this tool to construct a second-countable locally compact Hausdorff étale minimal groupoid $\cG$ with compact unit space such that
\begin{enumerate}
    \item the set $\Am(\cG)$ is not closed in $\Sub(\cG)$;
    \item\label{property:exoticity} there exists $x\in\cG^{(0)}$ such that
    \[
        \cs_e(\cG_x^x)\neq \cs_r(\cG_x^x).
    \]
\end{enumerate}
In particular, property~\ref{property:exoticity} gives a positive answer to \cite[Question 4.5]{ChristensenNeshveyev2}. Nevertheless, the reduced groupoid $\cs$-algebra of $\cG$ will be simple, hence Question \ref{qst:onepoint} remains open.

We define the free product of groupoids with common unit space. For a more general notion see \cite[Chapter 11]{Higgins}.

\begin{defn}
    Let $\cG$ and $\cH$ be groupoids with common unit space $X$.  Let $\cG^\times=\cG\setminus X$ and $\cH^\times=\cH\setminus X$. For an alternating tuple $\epsilon=(\epsilon_1,\ldots,\epsilon_l)\in\{ \cG^\times,\cH^\times\}^l$, let
\[
    W_\epsilon=\left\{(g_1,\ldots,g_l)\in\epsilon_1\times\cdots\times\epsilon_l:s(g_i)=r(g_{i+1})\text{ for }1\leq i<l\right\}.
\]
The \emph{free product} of $\cG$ and $\cH$ \emph{amalgamated over} $X$ is the groupoid
\[
\cG*_X\cH\defeq X\sqcup\bigsqcup_{\epsilon}W_\epsilon, 
\]
where $\epsilon$ ranges over the alternating tuples in $\{\cG^\times,\cH^\times\}^l$ and $l\in \{1,2,\dots\}$.

Two words $a_1\cdots a_n$ and $b_1\cdots b_m$ are composable if $r(b_1)=s(a_n)$ and in this case the composition is $a_1\cdots a_nb_1\cdots b_m$, followed by potential reductions. The inverse of $a_1\cdots a_n$ is $a_n\inv\cdots a_1\inv$. Finally elements of $X$ act as identities. With this structure $\cG*_X\cH$ is a groupoid with unit space $X$.
\end{defn}

We point out that the original definition of free products is given via a universal property. This more concrete version is essentially \cite[Chapter 11, Theorem 5]{Higgins}.

When $\cG$ and $\cH$ are Hausdorff étale groupoids, the free product $\cG*_X\cH$ carries a natural topology that turns it into a Hausdorff étale groupoid. Indeed, each $W_\epsilon$ is a (closed) subset of a product space and therefore carries a subspace topology; moreover, the decomposition $\cG*_X\cH\defeq X\sqcup\bigsqcup_{\epsilon}W_\epsilon$ induces the direct sum topology on $\cG*_X\cH$. More concretely, a subset $U\subseteq \cG*_X\cH$ is open if and only if $U\cap X$ is open in $X$ and for every alternating tuple $\epsilon$ the intersection $U\cap W_\epsilon$ is open in $W_\epsilon$. In particular, $X$ and all the $W_\epsilon$ are open in $\cG*_X\cH$.

\begin{lem}\label{lem:fpgroupoids}
    Let $\cG$ and $\cH$ be Hausdorff étale groupoids. Then the free product $\cG*_X\cH$ is a Hausdorff étale groupoid. If $\cG$ and $\cH$ are both locally compact and second countable, then so is $\cG*_X\cH$.
\end{lem}
\begin{proof}
    The unit space of a Hausdorff groupoid is closed, and that of an étale groupoid is open. Thus $X$ is clopen in both $\cG$ and $\cH$.

    By definition of the direct sum topology, if a net of reduced words converges to a word in $W_\epsilon$, then it belongs eventually to $W_\epsilon$ and converges coordinatewise there. Consequently, the source and range maps are continuous, since they are given on reduced words by
    \[
        s(a_1\cdots a_n)=s(a_n),
        \qquad
        r(a_1\cdots a_n)=r(a_1),
    \]
    and inversion is continuous because it reverses the order of the letters and applies inversion coordinatewise.

    To prove continuity of multiplication, let $(u_\lambda,v_\lambda)\to(u,v)$ be a net of composable pairs. After restricting to a tail, the lengths and factor patterns of $u_\lambda$ and $v_\lambda$ agree with those of $u$ and $v$, and all their letters converge coordinatewise. Consider the finite reduction procedure used to compute $uv$. At each step, either the two letters at the junction belong to different factors, in which case reduction stops, or they belong to the same factor and are multiplied there. The corresponding products for $u_\lambda v_\lambda$ converge to the product appearing in the reduction of $uv$. Since $X$ is clopen, such a product belongs to $X$ if and only if the corresponding products for the net eventually belong to $X$. Thus, by induction over the finitely many reduction steps, the reduction of $u_\lambda v_\lambda$ eventually performs exactly the same multiplications and cancellations as the reduction of $uv$. The letters of the resulting reduced words therefore converge coordinatewise to those of $uv$. If both words cancel completely, then $u_\lambda v_\lambda$ is eventually a unit and converges to $uv$ by continuity of the range map. Hence $u_\lambda v_\lambda\to uv$, proving continuity of multiplication.

    Each $W_\epsilon$ is Hausdorff, and therefore their topological disjoint union with $X$ is Hausdorff. It remains to verify étaleness. Let $a_1\cdots a_n\in W_\epsilon$, and choose open bisections $U_i$ in the appropriate factors such that $a_i\in U_i\subseteq\epsilon_i$. Then
    \[
        U_1\cdots U_n
        =
        \{u_1\cdots u_n\in W_\epsilon:u_i\in U_i\}
    \]
    is an open neighborhood of $a_1\cdots a_n$. A word in this set is uniquely determined by its source: starting from $x=s(u_n)$, its letters are recovered recursively by
    \[
        u_n=(s|_{U_n})^{-1}(x),
        \qquad
        u_i=(s|_{U_i})^{-1}(r(u_{i+1}))
        \quad(i=n-1,\ldots,1).
    \]
    Since the $U_i$ are open bisections, the set of points for which this recursion is defined is open, and the recursion depends continuously on $x$. Thus the source map restricts to a homeomorphism from $U_1\cdots U_n$ onto an open subset of $X$. The same holds at units because $X$ is open and $s|_X=\id_X$. Hence $\cG*_X\cH$ is étale.

    Finally, suppose that $\cG$ and $\cH$ are locally compact and second-countable. Each $W_\epsilon$ is closed in the corresponding finite product. It is therefore locally compact and second-countable. Since there are only countably many finite alternating tuples, their topological disjoint union with $X$ is locally compact and second-countable.
\end{proof}

We now turn to the answer to \cite[Question 4.5]{ChristensenNeshveyev2}. The construction is based on the groupoids introduced by Higson, Lafforgue and Skandalis in their counterexamples to the Baum--Connes conjecture, see \cite[Section 2]{HLS}. Christensen and Neshveyev used a related construction in \cite[Example 2.11]{ChristensenNeshveyev1} to exhibit exotic isotropy completions. We use a free product construction to make the resulting groupoid minimal. 

Let $\N_\infty=\N\cup\{\infty\}$ be the one-point compactification of $\N$, and let $C$ be the Cantor space. Choose a decreasing sequence $(\Gamma_n)_{n\in\N}$ of finite-index normal subgroups of $\Gamma=\mathbb F_2$ such that
\[
    \bigcap_{n\in\N}\Gamma_n=\{e\}.
\]
Set $K_n=\Gamma/\Gamma_n$, and denote the quotient map by $\pi_n\colon\Gamma\to K_n$. For notational simplicity, let $K_\infty=\Gamma$ and $\pi_\infty=\id_\Gamma$. 

Consider the group bundle
\[
    \mathcal B=\bigsqcup_{n\in\N_\infty}\bigl(\{n\}\times K_n\bigr)
\]
over $\N_\infty$. Every point $(n,k)$ with $n\in\N$ is isolated, while a neighbourhood basis of $(\infty,g)$ is given by
\[
    V_{g,N}=\{(\infty,g)\}\cup\{(n,\pi_n(g)):n\geq N\},\qquad N\in\N.
\]
The groupoid operations are defined fibrewise. It is straightforward to verify that $\mathcal B$ is second-countable, locally compact and étale. Moreover, the assumption $\bigcap_n\Gamma_n=\{e\}$ implies that $\mathcal B$ is Hausdorff.

Let $\mathcal C=\mathcal B\times C$. This is a group bundle over $X=\N_\infty\times C$ where the source and range of $(n,k,c)$ are both equal to $(n,c)$. The groupoid $\mathcal C$ is also second-countable, locally compact, Hausdorff and étale, and its unit space $X$ is homeomorphic to the Cantor space.

 Since $X$ is again a Cantor space, we can fix a minimal homeomorphism $T\colon X\to X$. Let $\mathcal H=X\rtimes_T\Z$ be the associated transformation groupoid. Let $\cG=\mathcal C*_X\mathcal H$ be the free product of $\mathcal C$ and $\mathcal H$ amalgamated over their common unit space. By Lemma \ref{lem:fpgroupoids}, $\cG$ is a second-countable locally compact Hausdorff étale groupoid.

 Since every arrow of $\mathcal C$ is isotropy, it follows that
\[
    \cG\cdot x=\{T^n(x):n\in\Z\}
\]
for every $x\in X$. Minimality of $T$ therefore implies that $\cG$ is minimal.

We first show that amenability is not closed in $\Sub(\cG)$. Fix $c\in C$ and for $n\in \N_\infty$ write $x_n=(n,c)$. We regard $K_n$ as the subgroup
\[
    \{(n,k,c):k\in K_n\}\leq\cG_{x_n}^{x_n}.
\]

\begin{prop}\label{prop:amenability-not-closed}
The sequence $(K_n)_{n\in \N}$ converges to $K_\infty$ in $\Sub(\cG)$. Consequently, $\Am(\cG)$ is not closed.
\end{prop}

\begin{proof}
We verify the touch-open and miss-compact conditions. Let $U\subseteq\cG$ be open and suppose that $U\cap K_\infty\neq\emptyset$. Choose $(\infty,g,c)\in U\cap K_\infty$. By the definition of the topology of $\cG$,
\[
    (n,\pi_n(g),c)\to(\infty,g,c).
\]
Since $(n,\pi_n(g),c)\in K_n$, it follows that $K_n\cap U\neq\emptyset$ for all sufficiently large $n$.

Now let $L\subseteq\cG$ be compact and suppose that $L\cap K_\infty=\emptyset$. Assume for a contradiction that $K_n\cap L\neq\emptyset$ for infinitely many $n$. After passing to a subsequence, choose strictly increasing $n_j$ and $k_j\in K_{n_j}\cap L$. Compactness of $L$ therefore allows us to pass to a further subsequence such that
\[
    k_j\to k\in L.
\]
The subgroupoid $\mathcal C$ is closed in $\cG$, so $k\in\mathcal C$. Furthermore,
\[
    s(k_j)=r(k_j)=x_{n_j}\to x_\infty.
\]
By continuity of source and range, $s(k)=r(k)=x_\infty$. The isotropy group of $\mathcal C$ at $x_\infty$ is precisely the embedded copy of $\Gamma$, so $k\in K_\infty$. This contradicts $L\cap K_\infty=\emptyset$.

Thus $K_n\to K_\infty$. Every $K_n$ is finite and therefore amenable, whereas $K_\infty =\mathbb F_2$ is nonamenable. Hence $\Am(\cG)$ is not closed.
\end{proof}

Next, we show that the isotropy norm at $x_\infty$ is exotic. For $n\in \N_\infty$, let $H_n=\cG_{x_n}^{x_n}$.

\begin{prop}\label{prop:minimal-exotic-fibre}
One has
\[
    \cs_e(H_\infty)\neq \cs_r(H_\infty).
\]
\end{prop}

\begin{proof}
    For $n\in \N$, the groups $K_n$ are amenable, hence their trivial representation $1_{K_n}$ is weakly contained in the left regular representation $\lambda_{K_n}$ \cite[Theorem G.3.2]{BHV}. Since induction preserves weak containment \cite[Theorem F.3.5]{BHV}, the induced representation $\sigma_n=\Ind_{K_n}^{H_n}(1_{K_n})$ is weakly contained in $\Ind_{K_n}^{H_n}(\lambda_{K_n})\cong \lambda_{H_n}$ \cite[Theorem E.2.4 and Example E.1.8]{BHV}. Explicitly, $\sigma_n$ is the quasi-regular representation of $H_n$ on $\ell^2(H_n/K_n)$ given by $\sigma_n(h)(\delta_{fK_n})=\delta_{hfK_n}$.  By \cite[Theorem F.4.4]{BHV}, we have the inequality $\|\sigma_n(a)\|\leq \|\lambda_{H_n}(a)\|=\|a\|_r$. Therefore $\sigma_n$ induces a homomorphism $\cs_r(H_n)\to \mathcal{B}(\ell^2(H_n/K_n))$ still denoted by $\sigma_n$. Let $\varphi_n$ be the state on $\cs_r(H_n)$ determined by $\delta_{K_n}$:
    \[   
    \varphi_n(a)=\langle\sigma_n(a)\delta_{K_n},\delta_{K_n}\rangle,
    \]
    for all $a\in \cs_r(H_n)$. For every $k\in K_n$, one has $\sigma_n(k)\delta_{K_n}=\delta_{K_n}$, and hence $\varphi_n(u_k)=1$. By \cite[Lemma~1.2]{ChristensenNeshveyev1}, the restriction map $C_c(\cG)\to \C[H_n]$ extends to a ucp map $\theta_{x_n,r}\colon \cs_r(\cG)\to \cs_r(H_n)$.

    Consider the state $\omega_n\colon \cs_r(\cG)\to \C$ given by $\omega_n=\varphi_n\circ\theta_{x_n,r}$. A simple computation shows that if $f\in C(X)$, then $\omega_n(f)=f(x_n)$. By weak-star compactness, up to passing to a subsequence (as $\cs_r(\cG)$ is separable), we can assume that $\omega_n\to \omega$ for a state $\omega$ on $\cs_r(\cG)$. 

    For $f\in C(X)$, we have 
    \[ \omega(f)=\lim_n\omega_n(f)=\lim_nf(x_n)=f(x_\infty). 
    \] 
    Therefore \cite[Proposition 1.11]{ChristensenNeshveyev2} implies the existence of a state $\varphi$ on $\cs_e(H_\infty)$ such that $\omega=\varphi\circ \theta_{x_\infty,e}$.
    We claim that $\varphi$ restricts to the trivial character on the embedded copy of $\C[\Gamma]\subseteq\C [H_\infty]$. Fix $g\in\Gamma$ and consider
    \[
    V_g=\{(n,\pi_n(g),d):n\in\N_\infty,\ d\in C\}.
    \]
    This is a compact open bisection of $\mathcal C$, and hence of $\cG$. Its restrictions to the relevant isotropy groups are $\theta_{x_n,r}(1_{V_g})=u_{\pi_n(g)}$ for $n\in\N$ and $\theta_{x_\infty,e}(1_{V_g})=u_g$.

    Therefore for $g\in \Gamma$ we obtain
    \begin{equation}\label{eq:limit-trivial-character}
        \begin{split}
            \varphi(u_g)&=\varphi(\theta_{x_\infty,e}(1_{V_g}))=\omega(1_{V_g})\\
            &=\lim_n \varphi_n(\theta_{x_n,r}(1_{V_g})) = \lim_n \varphi_n(u_{\pi_n(g)})=1.
        \end{split}   
    \end{equation}
    
Let $a,b$ be the standard free generators of $\Gamma=\mathbb F_2$ and put
\[
    m=\frac14\bigl(u_a+u_{a^{-1}}+u_b+u_{b^{-1}}\bigr)\in\C[\Gamma]\subseteq\C [H_\infty].
\]
Equation \eqref{eq:limit-trivial-character} gives
\[
    \|m\|_{\cs_e(H_\infty)}\geq|\varphi(m)|=1.
\]
Since $m$ is the average of four unitaries, the reverse inequality follows from the triangle inequality. Therefore $\|m\|_{\cs_e(H_\infty)}=1$.

On the other hand, as a representation of the subgroup $\Gamma$, the restriction of $\lambda_{H_\infty}$ is a direct sum of copies of $\lambda_\Gamma$. Indeed,
\[
    \ell^2(H_\infty)=\bigoplus_{\Gamma h\in\Gamma\backslash H_\infty}\ell^2(\Gamma h),
\]
and each summand is $\Gamma$-invariant and unitarily equivalent to $\ell^2(\Gamma)$. It follows that
\[
    \|m\|_{\cs_r(H_\infty)}=\|m\|_{\cs_r(\Gamma)}.
\]
Kesten's computation of the norm of the simple random walk operator on $\mathbb F_2$ gives
    \[
    \|m\|_{\cs_r(H_\infty)}=\|m\|_{\cs_r(\Gamma)}=\frac{\sqrt3}{2},
    \]
see \cite[Theorem~3]{Kesten} or \cite[Theorem IV.J]{AkemannOstrand}. Consequently, the canonical quotient $q_{x_\infty}$ cannot be an isomorphism.
\end{proof}

We conclude the section by showing that $\cs_r(\cG)$ is simple, hence it cannot be used to answer Question \ref{qst:onepoint}.

\begin{prop}
    For every unit point $x\in X$, the isotropy group $\cG_x^x$ is $\cs$-simple. Therefore $\cs_r(\cG)$ is simple.
\end{prop}

\begin{proof}
Since $T$ is minimal, the induced action of $\Z$ is free. Hence the transformation groupoid $\mathcal H=X\rtimes_T\Z$ is principal. Fix $x\in X$ and write
\[
    \mathcal O_T(x)=\{T^k(x):k\in\Z\}.
\]
For every $y\in\mathcal O_T(x)$, let $\eta_y\in\mathcal H$ be the unique arrow with source $x$ and range $y$. Conjugation by $\eta_y$ gives an embedding
\[
    \alpha_y\colon\mathcal C_y^y\longrightarrow\cG_x^x,
    \qquad
    \alpha_y(c)=\eta_y^{-1}c\eta_y.
\]
We claim that $\cG_x^x\cong \mathop{*}_{y\in\mathcal O_T(x)}\mathcal C_y^y$.

Put $P_x=\mathop{*}_{y\in\mathcal O_T(x)}\mathcal C_y^y$ and denote by $\iota_y\colon\mathcal C_y^y\to P_x$ the canonical inclusion. By the universal property of the free product, there is a unique group homomorphism
\[
    \Phi_x\colon P_x\longrightarrow\cG_x^x
\]
such that $\Phi_x\circ\iota_y=\alpha_y$ for every $y\in\mathcal O_T(x)$.

We first prove that $\Phi_x$ is injective. Let
\[
    \iota_{y_1}(c_1)\cdots\iota_{y_m}(c_m)
\]
be a nonempty reduced word in $P_x$. Thus $c_i\in\mathcal C_{y_i}^{y_i}\setminus\{y_i\}$ and $y_i\neq y_{i+1}$ for every $i<m$. Its image under $\Phi_x$ is
\[
\begin{split}
    \Phi_x\bigl(\iota_{y_1}(c_1)\cdots\iota_{y_m}(c_m)\bigr)
      &=
    \eta_{y_1}^{-1}c_1\eta_{y_1}
    \eta_{y_2}^{-1}c_2\eta_{y_2}
    \cdots
    \eta_{y_m}^{-1}c_m\eta_{y_m} \\
      &=
    \eta_{y_1}^{-1}c_1
    \bigl(\eta_{y_1}\eta_{y_2}^{-1}\bigr)c_2
    \cdots
    \bigl(\eta_{y_{m-1}}\eta_{y_m}^{-1}\bigr)c_m\eta_{y_m}.
\end{split}
\]
Since $y_i\neq y_{i+1}$, each arrow $\eta_{y_i}\eta_{y_{i+1}}^{-1}$ is a nonunit arrow of $\mathcal H$. After omitting $\eta_{y_1}^{-1}$ when $y_1=x$ and $\eta_{y_m}$ when $y_m=x$, the expression above is therefore a nonempty reduced word in the groupoid free product $\mathcal C*_X\mathcal H$. By uniqueness of reduced words, it is not a unit. Hence $\Phi_x$ is injective.

We now prove surjectivity. Consider the following subgroupoid of the reduction $\cG|_{\mathcal O_T(x)}$:
\[
    \mathcal R_x
       =
    \left\{
        \eta_y\Phi_x(p)\eta_z^{-1}
        \colon
        y,z\in\mathcal O_T(x),\ p\in P_x
    \right\}.
\]
Indeed, if the two arrows below are composable, then
\[
    \bigl(\eta_y\Phi_x(p)\eta_z^{-1}\bigr)
    \bigl(\eta_z\Phi_x(q)\eta_w^{-1}\bigr)
      =
    \eta_y\Phi_x(pq)\eta_w^{-1},
\]
and
\[
    \bigl(\eta_y\Phi_x(p)\eta_z^{-1}\bigr)^{-1}
      =
    \eta_z\Phi_x(p^{-1})\eta_y^{-1}.
\]

The subgroupoid $\mathcal R_x$ contains both factors restricted to $\mathcal O_T(x)$. If $c\in\mathcal C_y^y$, then
\[
    c
      =
    \eta_y\Phi_x\bigl(\iota_y(c)\bigr)\eta_y^{-1}
      \in
    \mathcal R_x.
\]
On the other hand, if $h\in\mathcal H$ has source $z$ and range $y$, then the principality of $\mathcal H$ gives
\[
    h=\eta_y\eta_z^{-1}
      =
    \eta_y\Phi_x(1)\eta_z^{-1}
      \in
    \mathcal R_x.
\]
Since $\cG|_{\mathcal O_T(x)}$ is generated by $\mathcal C|_{\mathcal O_T(x)}$ and $\mathcal H|_{\mathcal O_T(x)}$, it follows that
\[
    \mathcal R_x=\cG|_{\mathcal O_T(x)}.
\]

In particular, every $g\in\cG_x^x$ can be written as
\[
    g=\eta_x\Phi_x(p)\eta_x^{-1}=\Phi_x(p)
\]
for some $p\in P_x$, because $\eta_x=x$. Thus $\Phi_x$ is surjective and hence an isomorphism.

To show $\cs$-simplicity of $\cG_x^x$, we write $P_x$ as a free product of groups with at least three elements, and apply \cite[Corollary 12 and Theorem 14]{deLaHarpe}. 

Choose distinct $n,m\in\N$ such that $|K_n|,|K_m|\geq3$. Since the
$T$-orbit of $x$ is dense and the sets $\{n\}\times C$ and
$\{m\}\times C$ are nonempty and open, there exist distinct
$y,z\in\mathcal O_T(x)$ whose first coordinates are respectively
$n$ and $m$. Thus $|\mathcal C_y^y|\geq3$ and $|\mathcal C_z^z|\geq3$. Writing
\[
    P_x = \mathcal C_y^y* (\mathop{*}_{w\in\mathcal O_T(x)\setminus\{y\}} \mathcal C_w^w),
\]
both free factors have at least three elements.

Theorem \ref{thm:mainproved} implies that $\cs_r(\cG)$ is simple.
\end{proof}


\end{document}